\documentclass[a4paper,english]{lipics-v2021}
\hideLIPIcs
\nolinenumbers

\usepackage{amsthm}
\usepackage{amsmath}
\usepackage{amssymb}
\usepackage{amsfonts}
\usepackage{epsfig}
\usepackage{graphics}
\usepackage{graphicx}
\usepackage{booktabs}
\usepackage{url}

\usepackage{latexsym}
\usepackage{psfrag}
\usepackage{enumerate}
\usepackage{here}
\usepackage{esvect}
\usepackage[update,prepend]{epstopdf}
\usepackage{anyfontsize}
\usepackage{url}
\usepackage{hyperref}

\usepackage{microtype}

\newtheorem*{observation*}{Observation}

\newcommand{\ie}{{i.e.}}

\newcommand{\alg}{\textsf{ALG}}
\newcommand{\opt}{\textsf{OPT}}

\newcommand{\conv}{{\rm conv}}
\newcommand{\interior}{{\rm Int}}

\newcommand{\dist}{{\rm dist}}

\newcommand{\area}{{\rm Area}}
\newcommand{\vol}{{\rm Vol}}

\newcommand{\NN}{\mathbb{N}} 
\newcommand{\ZZ}{\mathbb{Z}} 
\newcommand{\RR}{\mathbb{R}} 
\newcommand{\eps}{\varepsilon}

\def\O{\mathcal O}
\def\P{\mathcal P}

\def\S{\mathcal S}

\newcommand{\later}[1]{}
\newcommand{\old}[1]{}

\title{Finding Points in Convex Position in Density-Restricted Sets}
\titlerunning{Finding Points in Convex Position in Density-Restricted Sets}

\author{Adrian Dumitrescu}
{Algoresearch L.L.C., Milwaukee, WI, USA}
{ad.dumitrescu@gmail.com}
{0000-0002-1118-0321}
{}
\author{Csaba D. T\'oth}
       {Department of Mathematics, California State University Northridge, Los Angeles, CA;
         and Department of Computer Science, Tufts University, Medford, MA, USA}
{csaba.toth@csun.edu}
{0000-0002-8769-3190}
{}
\authorrunning{A.~Dumitrescu and Cs.~D.~T\'oth}

\Copyright{Adrian Dumitrescu and Csaba D. T\'oth}



\begin{document}

\maketitle

\begin{abstract}
  For a finite set $A\subset \RR^d$, let $\Delta(A)$ denote the spread of $A$, which is the ratio of
  the maximum pairwise distance to the minimum pairwise distance. For a positive integer $n$,
  let $\gamma_d(n)$ denote the largest integer such that any set $A$ of $n$ points in general position in $\RR^d$,
  satisfying $\Delta(A) \leq \alpha n^{1/d}$ for a fixed $\alpha>0$,
  contains at least $\gamma_d(n)$ points in convex position.
  About $30$ years ago, Valtr proved that $\gamma_2(n)=\Theta(n^{1/3})$. Since then
  no further results have been obtained in higher dimensions. Here we continue this line of research
  in three dimensions and prove that $ \gamma_3(n) =\Theta(n^{1/2})$.
   The lower bound implies the following approximation:
   Given any $n$-element point set $A\subset \RR^3$ in general position, satisfying $\Delta(A) \leq \alpha n^{1/3}$
   for a fixed $\alpha$, a $\Omega(n^{-1/6})$-factor approximation of the maximum-size convex subset of points
   can be computed by a randomized algorithm in $\O(n \log{n})$ expected time.

%
\end{abstract}

\section{Introduction} \label{sec:intro}

A set of points in the $d$-dimensional space $\RR^d$ is said to be:
(i)~in \emph{general position} if any $d+1$ or fewer points are \emph{affinely independent}; and
(ii)~in \emph{convex position} if none of the points lies in the convex hull of the other points.

In $1935$ Erd\H{o}s and Szekeres proved, as one of the first Ramsey-type results in combinatorial geometry,
that for every $n \in \NN$, a sufficiently large point set in the plane in general position contains
$n$ points in convex position~\cite{ES35}. The minimum cardinality of a point set that contains a subset of $n$ points
in convex position is known as the Erd\H{o}s--Szekeres number, denoted here by $f(n)$.
The resulting upper bound in their paper was $f(n) \leq \binom{2n-4}{n-2}+1=4^{n(1-o(1))}$.
In $1960$, the same authors showed a construction that implies a lower bound of $f(n) \geq 2^{n-2}+1$,
and conjectured that this lower bound is tight~\cite{ES60}.
The current best (asymptotic) upper bound, due to Suk~\cite{Suk17}, is $f(n) \leq 2^{n(1+o(1))}$.
In other words, every set of $n$ points in general position in the plane contains $(1-o(1))\log{n}$ points in convex position,
and this bound is tight up to lower-order terms.

Let $A$ be a set of $n$ points in general position in $\RR^d$.
Define
\[
\Delta(A)= \frac{\max\{\dist(a,b) : a,b \in A, a \neq b\}}{\min\{\dist(a,b) : a,b \in A, a \neq b\}}, \]
where $\dist(a,b)$ is the Euclidean distance between points $a$ and $b$.
The ratio $\Delta(A)$ is referred to as the \emph{aspect ratio} or the \emph{spread} of $A$;
see for instance~\cite{BLMN05,CL21}.
We assume without loss of generality that the minimum pairwise distance is $1$ and in this case
$\Delta(A)$ is the diameter of $A$.
A standard volume argument shows that if $A$ has $n$ points, then $\Delta(A) \geq c_d \, n^{1/d}$, where
$c_d>0$ is a constant depending only on $d$; for instance it is known~\cite{Va92}
that $c_2 \geq 2^{1/2} 3^{1/4} \pi^{-1/2} \approx 1.05$.
On the other hand, the section of the integer lattice $[n]^d$  shows that this bound is tight
up to the aforementioned constant.
A point set satisfying the condition $\Delta(A) =\O(n^{1/d})$, is called here \emph{density-restricted}
(or simply \emph{dense}, see for instance \cite{EVW97,KT20,Va94}).

In the seminal article of Erd\H{o}s and Szekeres~\cite{ES35}, the constructed point sets
with no large subsets in convex position have very large spread.
Similarly, in Horton's seminal article on point sets with no empty convex heptagons,
the constructed point sets also have very large spread
(Horton sets will be discussed in Section~\ref{sec:prelim}.)
Answering the emerging question of whether such results really require large spreads,
Valtr~\cite{Va92} showed the existence of arbitrarily large planar sets $A$ with $\Delta(A) \leq \alpha n^{1/2}$
that have no empty convex heptagons.
(Observe that whenever $\Delta(A)$ satisfies this condition, we have $\alpha \geq c_2$.)
This property can be achieved for example by a suitable, carefully crafted, small perturbation
of the lattice section $[n]^2$. Moreover, Valtr~\cite{Va92} obtained by probabilistic arguments the following result:

\begin{theorem} {\rm (Valtr~\cite{Va92})} \label{thm:valtr-lower}
  For every $\alpha>0$ there exists $\beta=\beta(\alpha)>0$ such that any
  set of $n$ points in general position in the plane satisfying $\Delta(A) \leq \alpha n^{1/2}$, contains a subset of
  $\beta n^{1/3}$ points in convex position.
\end{theorem}

It is therefore apparent that the size of the largest convex subset contained in an $n$-element point set
strongly depends on the spread of the set. In the same paper, Valtr generalized Theorem~\ref{thm:valtr-lower}
to a broader range of the spread (by similar arguments):

\begin{theorem} {\rm (Valtr~\cite{Va92})}  \label{thm:valtr-lower-general}
  For every $\alpha>0$ and $\tau \in \left[\frac12, \frac34\right)$, there exists $\beta=\beta(\alpha,\tau)>0$
    such that any set $A$ of $n$ points in general position in the plane satisfying
    $\Delta(A) \leq \alpha n^\tau$ contains a subset of $\beta n^{1- 4\tau/3}$ points in convex position.
\end{theorem}

On the other hand, as remarked by Valtr, the construction by Erd\H{o}s and Szekeres~\cite{ES60} can
be transformed into one with spread $\Delta<n$ and with at most $\log{n} +2$ points in convex position.
For the remaining range $\tau \in \left[\frac34, 1 \right)$, Valtr stated in his PhD Thesis~\cite{Va94}
(without proof) that an involved argument shows that any set $A$ of $n$ points in general position
in the plane satisfying $\Delta(A) \leq \alpha n^\tau$ (where $\alpha>0$ is a constant)
contains a convex subset of size $\beta n^{\eps}$ for some $\eps=\eps(\tau)>0$.
We will return to this question in Section~\ref{sec:conclusion}.

For density-restricted sets, the previously mentioned small perturbation of the lattice section $[n]^2$ provides
a matching upper bound:

\begin{theorem} {\rm (Valtr~\cite{Va92})} \label{thm:valtr-upper}
  For every $n \in \NN$ there exists an $n$-element point set $A \subset \RR^2$ in general position,
  satisfying $\Delta(A) \leq \alpha n^{1/2}$, for some constant $\alpha>0$,
 in which every subset in convex position has at most $\O(n^{1/3})$ points.
\end{theorem}

For a positive integer $n$, let $\gamma_d(n)$ denote the largest integer such that any set $A$ of $n$ points
in general position in $\RR^d$, satisfying $\Delta(A) \leq \alpha n^{1/d}$  for a fixed $\alpha>0$,
contains at least $\gamma_d(n)$ points in convex position. In these terms, Valtr's result is that
$\gamma_2(n)=\Theta(n^{1/3})$. The estimation of $\gamma_d(n)$ for $d \geq 3$ has remained an open problem.
Here we continue this line of research in three dimensions
and prove that $\gamma_3(n)=\Theta(n^{1/2})$.

\medskip\noindent\textbf{Our Results.}

\begin{theorem} \label{thm:lower}
  For every $\alpha>0$ there exists $\beta(\alpha)>0$ such that
  any set of $n$ points in $\RR^3$ in general position, satisfying $\Delta(A) \leq \alpha n^{1/3}$,
  contains a subset of at least $\beta n^{1/2}$ points in convex position.
 In particular, $\gamma_3(n)=\Omega(n^{1/2})$.
\end{theorem}

\begin{theorem} \label{thm:lower-general}
  For every $\alpha>0$ and $\tau \in \left[\frac13, \frac23\right)$, there exists $\beta=\beta(\alpha,\tau)>0$
    such that any set $A$ of $n$ points in general position in $\RR^3$ satisfying
    $\Delta(A) \leq \alpha n^\tau$ contains a subset of $\beta n^{1- 3\tau/2}$ points in convex position.
\end{theorem}

\begin{theorem} \label{thm:upper}
  For every $n \in \NN$ there exists an $n$-element point set $A \subset \RR^3$ in general position,
  satisfying $\Delta(A) \leq \alpha n^{1/3}$, for some constant $\alpha>0$,
 in which every subset in convex position has at most $\O(n^{1/2})$ points.
  In particular, $\gamma_3(n)=\O(n^{1/2})$.
\end{theorem}

\begin{theorem} \label{thm:approx}
Given any $n$-element point set $A\subset \RR^3$ in general position, satisfying $\Delta(A) \leq \alpha n^{1/3}$
for a fixed $\alpha$, a $\Omega(n^{-1/6})$-factor approximation of the maximum-size convex subset of points
can be computed by a randomized algorithm in $\O(n \log{n})$ expected time.
\end{theorem}



In 1960, Erd\H{o}s and Szekeres~\cite{ES60} constructed for every positive integer $k$,
a set of $n:=2^k$ points in general position in the plane, such that the size of the largest convex subset
is $k-1=\log{n}-1$. In 2017, Duque, Fabila-Monroy, and Hidalgo-Toscano~\cite{DFH18} showed how to realize
that construction on an integer grid of size $\O(n^2 \log^3{n})$.
In 1978, Erd\H{o}s asked whether, given any positive integer $k$,
every sufficiently large point set in general position in the plane contains $k$ points in convex position
such that the respective polygon is empty of other points.
In 1983, Horton~\cite{Hor83} gave  a negative answer by constructing arbitrarily large point sets with
no empty convex $7$-gon. Such sets are generally called Horton sets. In 2017,
Barba, Duque, Fabila-Monroy, and Hidalgo-Toscano~\cite{BDFH17} showed how to realize a Horton set of size $n$
on an integer grid of size $\O(n^{\frac12 \log{(n/2)}})$. On the other hand, they proved that any set of $n$ points
with integer coordinates combinatorially equivalent to a Horton set contains a point with a coordinate at least
$\Omega(n^{\frac{1}{24} \log{(n/2)}})$.

  Given a point set in general position in $\RR^d$, the problem of computing a maximum-size
  subset in convex position can be solved in polynomial time for $d=2$ by the dynamic programming
  algorithm of Chv\'atal and Klincsek~\cite{CK80}; their algorithm runs in $\O(n^3)$ time.
  In contrast, the general problem in $\RR^d$ was shown to be \textrm{NP}-complete for every $d \geq 3$
  by Giannopoulos, Knauer, and Werner~\cite{GKW13}.
  
Several problems concerning convex polygons (resp., polytopes) whose vertices lie in a Cartesian product
of two (resp., $d$) sets of reals have been recently studied in~\cite{CDM+20}.
See also \cite[Sec.~8.2]{BMP05} for additional problems related to Horton sets and the Erd\H{o}s-Szekeres theorem.
A comprehensive survey on the Erd\H{o}s--Szekeres problem is due to Morris and Soltan~\cite{MS16}.

Recently, Bukh and Dong~\cite{BD22} independently proved that $\gamma_d=\Theta_d(n^{(d-1)/(d+1)})$ for all $d\in \NN$, using more or less different techniques.

\smallskip\noindent\textbf{Definitions and notations.}
For a finite point set $S \subset \RR^d$, let $g(S)=g_d(S)$ be the maximum size
of a convex subset of $S$; when there is no confusion, the subscript $d$ may be omitted.
Let $g(n)$ be defined as
\[ g(n) =\min \{g(S) \colon |S|=n, S \text{ is in general position}\}. \]
The \emph{interior} and the \emph{boundary} of a set $S \subset \RR^d$
are denoted by $\interior(S)$ and $\partial S$, respectively.
A vector in $\vec{v} = (x_1,\ldots,x_d) \in \ZZ^d$ is \emph{primitive} if $\gcd(x_1,\ldots,x_k)=1$.
Let $[n]$ denote the set $\{1,2,\ldots,n\}$.
Unless specified otherwise, all logarithms are in base $2$.

Here we use the convention that the approximation ratio of an algorithm is $<1$ for a maximization problem
and $>1$ for a minimization problem (as in~\cite{WS11}).
We frequently write $\O_d$, when needed, to indicate that the hidden constant in the $\O$ asymptotic notation
depends only on $d$.

\section{Preliminaries}  \label{sec:prelim}

A set of points in the plane $\{p_i=(x_i,y_i), i=1,\ldots,n\}$ is in \emph{strong general position}
if it is in general position and no two $x$- or $y$-coordinates are the same.
More generally, a set of points $P \subset \RR^d$ $\{p_i, i=1,\ldots,n\}$ is said to be in \emph{strong general position}
if $P$ is in general position and no two values of the $j$th coordinate are the same, for $j=1,\ldots,d$.
For any $f\in \{0,1,\ldots , d-1\}$, denote by $\pi_f$ the orthogonal projection from $\RR^d$ onto $\RR^f$
obtained by retaining the first $f$ coordinates.

Let $A,B\subset \RR^d$ be finite sets such that $A\cup B$ is in strong general position in $\RR^d$.
We say that $A$ lies \emph{deep below} $B$ (equivalently, $B$ lies  \emph{high above} $A$)
if the following holds:
every point in $A$ lies below all hyperplanes determined by $d$ points in $B$ and
every point in $B$ lies above all hyperplanes determined by $d$ points in $A$
(w.r.t.\ the $d$th coordinate)~\cite[Ch.~3.2]{Mat02}.
Whenever needed, we extend this relation to apply to any coordinate other than the last one.

Let $X\subset \RR^d$ be a finite set in general position. A $k$-element set $Y \subseteq X$ is called
a \emph{$k$-hole} in $X$ if $Y$ is in convex position and $\conv(Y) \cap X=Y$~\cite[Ch.~3.2]{Mat02}. In the plane
($d=2$), $Y$ is the vertex set of a convex polygon with no point of $X$ inside.
Horton~\cite{Hor83}, answering a question by Erd\H{o}s, constructed arbitrary large finite sets without a $7$-hole.
On the other hand, every sufficiently large finite set in general
position contains six points that determine a $6$-hole (\ie, an empty convex hexagon);
this was shown about 25 years later by Gerken~\cite{Ger08} and respectively Nicol\'as~\cite{Nic07}.

\smallskip\noindent\textbf{Horton sets and Valtr grids in the plane.}
For a sequence of points sorted by $x$-coordinates (with no duplicate $x$-coordinates) $P=\{p_1,p_2,\ldots,p_n\}$, let
$P_0=\{p_2,p_4,\ldots\}$ denote the subsequence of points with even indexes and
$P_1=\{p_1,p_3,\ldots\}$ denote the subsequence of points with odd indexes.
A finite set $H$ is a \emph{Horton set} if either (i)~$|H|=1$, or (ii)~$|H| \geq 2$,
both $H_0$ and $H_1$ are Horton sets and one of these two sets lies deep below the other one.

 The above definition allows one to construct inductively Horton sets
in the plane of any size; see~\cite[Ch.3.2]{Mat02}. Following~\cite{CL21,Hor83,Va92}, one can
construct Horton sets as follows. For a nonnegative integer $N$, let
its binary representation be $N=\sum_{i \geq 0} a_i 2^i$, where $a_i \in \{0,1\}$.
For any positive $0< \eps < 1/2$, denote by $(N)_\eps$
the real number
\begin{equation} \label{eq:(N)_eps}
  (N)_\eps = \sum_{i \geq 0} a_i \eps^{i+1},
\end{equation}
and note that $0< (N)_\eps < 2\eps$. Then the set
$S= \{p_i : i \in [n]\}, \text{ where } p_i = (i,(i)_\eps)$,
is a Horton set. Indeed, $S_0$ lies deep below $S_1$ if $\eps$ is sufficiently small, and further on,
this property holds recursively for the smaller sets  $S_0$ and $S_1$.

Interestingly enough, Valtr~\cite{Va92} showed how to use multiple Horton sets to obtain
planar sets that are sufficiently small perturbations of a lattice section $[n]^2$ and that preserve
the key property of the Horton set, namely that of not containing any $7$-hole.
The resulting sets (described in Theorem~\ref{thm:valtr-upper}) are not Horton sets per se.

\smallskip\noindent\textbf{Higher dimensions.}
Recently, Conlon and Lim~\cite{CL21} extended Valtr's planar construction to higher dimensions
such that the resulting sets do not contain large holes (\ie, convex polytopes with a large number of vertices
that are empty of points from the set).
Their result is summarized in Theorem~\ref{thm:cl-grid}.

\begin{theorem} {\rm (Conlon and Lim~\cite{CL21})} \label{thm:cl-grid}
  For any integers $n \geq 1$ and $d \geq 2$ and any $\eps>0$, there exists an integer $C_d=d^{\O(d^3)}$
  and a set of points $\{P_{\vec{x}} \mid \vec{x} \in [n]^d\} \subset \RR^d$ that is $C_d$-hole-free and satisfies
    $\dist(P_{\vec{x}}, \vec{x}) \leq \eps$ for all $\vec{x} \in [n]^d$.
\end{theorem}

Both the \emph{Valtr grid}~\cite{Va92} and the \emph{Conlon-Lim grid}~\cite{CL21}
can be described in terms of negligible perturbations (or negligible functions).
Given two finite sets $A,B \subset \RR^d$, a bijection $f \colon A \to B$ is \emph{negligible} if,
whenever $S \subset A$ and $x \in A$ with $x \in \interior( \conv(S))$, then $f(x) \in \interior (\conv(f(S)))$.

The planar construction can be described in one paragraph as follows (we mainly follow here the outline
of Conlon and Lim~\cite{CL21}). Starting with the lattice section $S_0=[n]^2$, each column of points
is shifted vertically by a small positive amount (\ie, negligible perturbation), resulting in a set $S_1$.
The shifting is done in such a way that each row of points is a Horton set. This also implies that
any nonvertical line of lattice points yields a subset of $S_1$ that is a Horton set.
Finally, each row of points is shifted horizontally by an even smaller positive amount
to get $S_2=S$ so that each column of points is a  Horton set. Overall, every line
of lattice points in $S_0$ corresponds to a Horton set in $S_2$. In other words, $S$ is the Minkowski sum of two
Horton sets, one of them resembling $[n]$ along the $x$-axis and the other along the $y$-axis.
It is easy to adjust the two perturbations if needed so that the final set $S$ is in strongly general position.
It is worth noting that the final set $S$ is \emph{not} a Horton set by itself.

\smallskip\noindent\textbf{Number of vertices and faces.}
The number of vertices of a convex lattice polytope can be bounded from above
using a theorem by Andrews~\cite{And63}.

\begin{theorem} {\rm (Andrews~\cite{And63}; see also~\cite{BP92,BV92,Sch85})} \label{thm:Andrews}
For every finite set $S\subset \ZZ^d$, $d\geq 2$, the lattice polytope $\conv(S)$ has
$\O_d \left(V^{\frac{d-1}{d+1}}\right)$ vertices, where $V=\vol(\conv(S))>0$.
\end{theorem}

The same bound holds for the total number of faces 
of the lattice polytope $\conv(S)$.
\begin{theorem} {\rm (\cite[Theorem~3.2]{Bar08} see also~\cite{BL98,KS84})} \label{thm:Andrews+}
  For every finite set $S\subset \ZZ^d$, $d\geq 2$, the lattice polytope $\conv(S)$ has
$\O_d \left(V^{\frac{d-1}{d+1}}\right)$ faces (of any dimension), where $V=\vol(\conv(S))>0$.
\end{theorem}

We generalize Theorem~\ref{thm:Andrews} in the plane in another direction, as follows.
This technical tool will be used in Section~\ref{sec:upper}.

\begin{lemma}\label{lem:new2}
For every finite set $S\subset \ZZ^2$ and for every integer $t$, $1\leq t\leq \sqrt{A}$,
the lattice polygon $\conv(S)$ has $\O\left((A/t^2)^{1/3}\right)$
edges that contain more than $t$ points in $\ZZ^2$, where $A=\area(\conv(S))>0$.
\end{lemma}
\begin{proof}
  For brevity, let $P=\conv(S)$. The boundary of $P$ can be decomposed into two $x$-monotone polygonal chains,
  a lower arc $P_1$ and an upper arc $P_2$. It suffices to show that $P_1$ and $P_2$ each have $\O((A/t^2)^{1/3})$
  edges that contain more than $t$ grid points.

Consider $P_1$ (the argument for $P_2$ is analogous and omitted).
Suppose that $P_1$ has $L_1$ edges that contain more than $t$ grid points.
Denote the vertex set of $P_1$ as $S_1=\{v_0,v_1,\ldots , v_k\}$,
labeled in left-to-right order, and assume that $v_0$ is the origin.
Note that $S_1\subset S$, and so $\area(\conv(S_1)) \leq \area(\conv(S))=A$.
For $i=1,\ldots ,k$, if the edge $v_{i-1}v_i$ contains $m_i$ grid points,
then $v_i-v_{i-1}=(m_i-1)\vec{e}_i$ for a primitive vector $\vec{e}_i$.
For every $i=1,\ldots, k$, let $m_i'$ be the largest multiple of $t$ such that $0\leq m_i'\leq m_i-1$.
Note that if $m_i \leq t$ then $m'_i =0$.

Let $S'_1$ be the vertex set of the polygonal arc obtained by concatenating the vectors
$m_1' \vec{e}_1,\ldots , m_k'\vec{e}_k$.
Since $m_i' \leq m_i$ for all $i$, then $\area(\conv(S_1')) \leq \area(\conv(S_1)) \leq A$.
By construction, we have $S_1' \subset t\,\ZZ^2$,
hence $S''_1: = \frac{1}{t}\, S_1'\subset \ZZ^2$.
Thus $\conv(S_1'')$ is a lattice polygon of area at most $A/t^2$. 
If $\conv(S_1'')$ has positive area, then Theorem~\ref{thm:Andrews} applies, otherwise $L_1\leq 2$.
In both cases, $\conv(S_1'')$ has $L_1 =\O((A/t^2)^{1/3})$ vertices, as required.
\end{proof}

\section{Lower bound: Proof of Theorem~\ref{thm:lower}}   \label{sec:lower}

Let $A$ be a set of $n$ points in $\RR^3$ in general position, satisfying $\Delta(A) \leq \alpha n^{1/3}$.
We may assume that $A \subset B$, where $B$ is a ball of radius $R= \alpha n^{1/3}$ centered at $O$;
see Fig.~\ref{fig:cap}\,(left). Consider any maximal packing $\P$ of spherical caps of height $h=\alpha n^{-1/6}$
and radius $r=(R^2-(R-h)^2)^{1/2}=\Theta(\alpha n^{1/12})$. For instance, $\P$ may be constructed greedily
from an empty packing by successively adding spherical caps that do not intersect previous caps.

\begin{figure}[htbp]
\centering
\includegraphics[width=.7\textwidth]{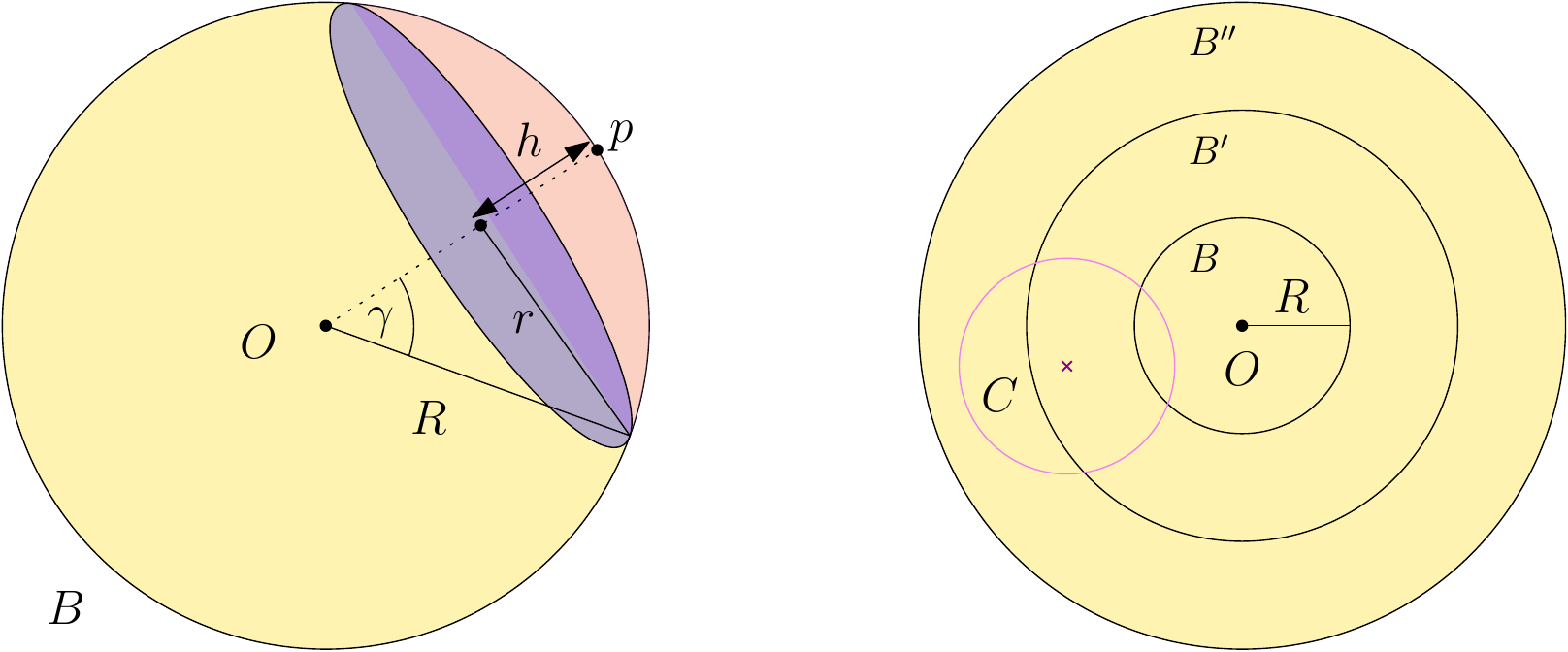}
\caption{Left: Elements of a spherical cap: the radius of the sphere $R$, the radius of the base of the cap $r$,
  the center $p$ of the cap, the height of the cap $h$, and the polar angle $\gamma$.
Right: Concentric balls and the placement of $C$.}
\label{fig:cap}
\end{figure}

\begin{lemma} \label{lem:packing}
  The spherical caps in $\P$ cover at least one quarter of the surface area of $B$.
\end{lemma}
\begin{proof}
  Consider the set $\P'$ of enlarged caps with the same centers as in $\P$ but with double polar angle.
  Due to the maximality of $\P$, the caps in $\P'$ cover the surface $\partial B$ of $B$.
  Indeed, if there is an uncovered point $p \in \partial B$, then the cap of spherical radius $r$
  centered at $p$ would not intersect any cap in $\P$, contradicting the maximality of $\P$.
  Doubling the polar angle of a spherical cap increases its area by a factor of at most 4.
  Consequently, the surface area of the union of spherical caps in $\P'$ is
  less than four times the surface area covered by $\P$.
  Since $\P'$ covers $\partial B$, then the caps in $\P$
  cover at least one quarter of the surface area of $B$.
\end{proof}

Since the area of $\partial B$ is $4 \pi R^2$ and the area of each spherical cap in $\P$ is $2 \pi R h$,
the cardinality of $\P$ is bounded from below (by Lemma~\ref{lem:packing}) as follows:
\begin{equation} \label{eq:P}
   |\P|  \geq \frac{4 \pi R^2}{4 \cdot 2 \pi R h} = \frac{R}{2h} = \frac{n^{1/2}}{2}.
\end{equation}
Label these caps as $\Psi_i$, $i=1,\ldots, |\P|$, in an arbitrary fashion.
The following is analogous with Lemma~2.2 in~\cite{Va92}, however,
due to the specifics of our construction, its proof is obvious.

\begin{lemma} \label{lem:convex}
  If we choose at most one point from each of the spherical caps $\Psi_i$, $i=1,\ldots, |\P|$, then the chosen
  points are in convex position.
\end{lemma}
We wish to find a suitable placement of a ball $C$ congruent to $B$ (in a rotated position)
so that a constant fraction of its spherical caps in $\P$ contains some points from $A$.
This would prove Theorem~\ref{thm:lower}.
As in Valtr's work, we proceed by probabilistic methods; the resulting algorithm is randomized.

Let $\S$ be the (infinite) set of all spherical caps of radius $r$ and height $h$ corresponding to all balls $C$
congruent to $B$ contained in a ball $B''$ concentric with $B$ of radius $3R$ centered at $O$;
note that any such ball $C$ is centered in a ball $B'$ concentric with $B$ of radius $2R$ centered at $O$.
Refer to Fig.~\ref{fig:cap}\,(right).
A spherical cap in $\S$ is determined by two parameters:
(1) the \emph{direction}, which is a unit vector in $\mathbb{S}^2$ parallel to the vector
from the center of the corresponding ball $C$ to the center of the cap; and
(2) the center of the corresponding ball $C$, which is in $B'$.
Consequently, $\S$ is parameterized by the Cartesian product $\mathbb{S}^2\times B'$; and
the Lebesgue measures on $B'$ and $\mathbb{S}^2$, resp., induce a product measure $\mu$ on $\S$.
Since we have $\vol(B')= \frac{4 \pi}{3} (2R)^3= \frac{32}{3}\pi R^3$ and $\area(\mathbb{S}^2)=4 \pi^2$,
then $\mu(\S)$ has a simple formula:
\begin{equation} \label{eq:mu}
\mu(\S)= \frac{128}{3}  \pi^3 R^3.
\end{equation}

A key lemma (analogous to Lemma~2.3 in~\cite{Va92}) is the following.

\begin{lemma} \label{lem:key}
There exists a positive constant $c=c(\alpha)$ such that
\[ \mu(\{S \in \S \colon S \cap A \neq \emptyset\}) \geq c \mu(\S).  \]
\end{lemma}

Its proof relies on two lemmas (analogous to Lemmas~2.4 and~2.5 in~\cite{Va92}).

\begin{lemma} \label{lem:shell}
  Let $a\in A$ and for every $j \in \NN$ let $A(j)$ be the set of points $a'\in A$ such that
  $j \leq \dist(a,a') \leq j+1$. Then $|A(j)|\leq c_1 j^2$, for some constant $c_1>0$.
\end{lemma}
\begin{proof}
  Draw $\O(j^2)$ planes incident to $a$ that divide the surface of the ball of radius $j+1$ centered at $a$ into
  $\O(j^2)$ surface patches of diameter at most $1/2$. The three concentric spheres of radii $j$, $j+1/2$, and $j+1$
  together with the $\O(j^2)$ planes partition the set $A(j)$ into $\O(j^2)$ $3$-dimensional cells of diameter
  less than $1$. As such, each cell contains at most one point from $A$ and the lemma follows.
\end{proof}

For every point $a \in A$, and every integer $i=0,1,\ldots,n$, we introduce the notation
\later{
$\S(a) = \{S \in \S \colon a \in S\}$,
$\S_i = \{S \in \S \colon |S \cap A|=i\}$, and
  $m_i = \mu(\S_i)$.
} 
%
\begin{align*}
  \S(a) &= \{S \in \S \colon a \in S\},\\
  \S_i &= \{S \in \S \colon |S \cap A|=i\}, \\
  m_i &= \mu(\S_i).
\end{align*}

\begin{lemma} \label{lem:locus}
  The locus of directions $\vec{v}\in \mathbb{S}^2$ corresponding to spherical caps in $\S(a) \cap \S(a')$
  is contained in a spherical ring of polar angle $\theta$, where
$\theta/2= \arcsin \frac{h}{\dist(a,a')}$; see Fig.~\ref{fig:locus3}.
\end{lemma}

\begin{figure}[htbp]
\centering
\includegraphics[width=.65\textwidth]{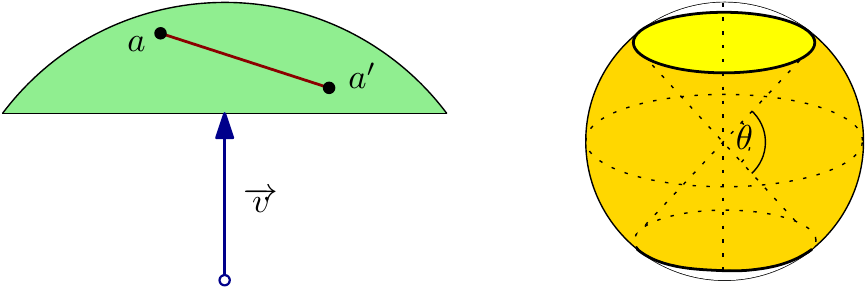}
\caption{A point pair $a,a'\in A$ in a spherical cap of direction $\vec{v}$;
and a spherical ring of angle $\theta$.}
\label{fig:locus3}
\end{figure}

\begin{proof}
  Let $\Psi$ be a spherical cap in $\S(a) \cap \S(a')$ and let $p$ denote its center.
  We distinguish between two cases.

  \emph{Case 1.} $1 \leq \dist(a,a') \leq \sqrt{r^2 + h^2}$.
  In an extremal position, $a'=p$ and $a$ is on the base of the cap.
  Refer to Fig.~\ref{fig:locus4}\,(left). The locus of directions $\vec{v}\in \mathbb{S}^2$ corresponding
  to spherical caps in $\S(a) \cap \S(a')$ is a spherical ring of polar angle $\theta$, where
  $\theta/2= \arcsin \frac{h}{\dist(a,a')}$.

\begin{figure}[htbp]
\centering
\includegraphics[width=0.7\textwidth]{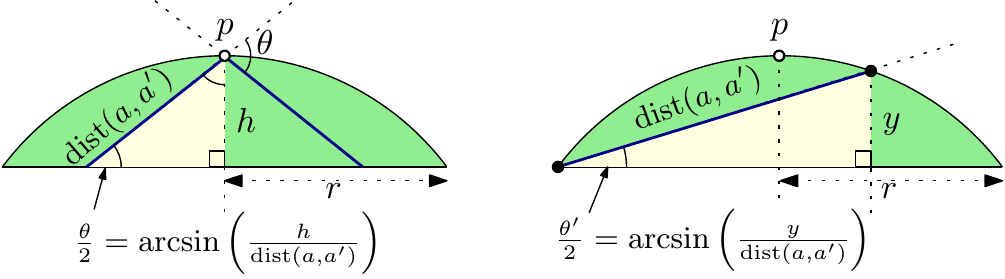}
\caption{Extremal positions of a segment that fits into the cap in Case~1 (left) and Case~2 (right).}
\label{fig:locus4}
\end{figure}

\emph{Case 2.} $\sqrt{r^2 + h^2}<\dist(a,a') \leq 2r$.
In an extremal position, $a$ is on the circle at the base of the cap, $a'$ is on the surface of the cap, and
  the great circle incident to $a$ and $a'$ passes through $p$; see Fig.~\ref{fig:locus4}\,(right).
  The locus of directions $\vec{v}\in \mathbb{S}^2$ corresponding to
  spherical caps in $\S(a) \cap \S(a')$ is a spherical ring of polar angle $\theta'$, where
  $\theta'/2=  \arcsin \frac{y}{\dist(a,a')} \leq \arcsin \frac{h}{\dist(a,a')}$,
  and $y$ is the distance between $a'$ and the base of $\Psi$.
  Consequently, this locus is contained in a spherical ring of polar angle $\theta$,
  where $\theta/2= \arcsin \frac{h}{\dist(a,a')}$.
  In particular, observe that when $\dist(a,a') =2r$, then $y=\theta'=0$.
 \end{proof}

\begin{lemma} \label{lem:mu}
  There exist positive constants $c_2$, $c_3$, and $c_4$ such that:
\begin{enumerate} [{\rm (i)}]
\item $\forall a \in A \colon c_2 \leq \mu(\S(a)) \leq c_3$, and
\item $\forall a \in A \colon \sum_{a' \in A} \mu(\S(a) \cap \S(a')) \leq c_4$.
\end{enumerate}
 \end{lemma}
\begin{proof}
  (i) The centers of spherical caps $S \in \S(a)$ with a fixed direction $\vec{v} \in \mathbb{S}^2$
  form a congruent spherical cap of the same direction $\vec{v}$.
  Hence $\mu(\S(a)) = v_0 \cdot \area(\mathbb{S}^2)$, where
  $v_0= \frac{\pi h^2}{3} (3R-h) = \pi \alpha^3 (1 - n^{-1/2}/3)$
  is the volume of a spherical cap. Since $1 - n^{-1/2}/3 \geq 5/6$ for $n \geq 4$,
and $\area(\mathbb{S}^2)=4 \pi^2$,
we deduce that $c_2 \leq \mu(\S(a)) \leq c_3$, for constants
$ c_2 = \pi \alpha^3 \cdot \frac56 \cdot 4 \pi^2$ and $ c_3 = \pi \alpha^3 \cdot 4 \pi^2$.

  (ii) For $a,a' \in A$, let $W_{a,a'}$ be the locus of
  directions $\vec{v}\in \mathbb{S}^2$ corresponding to spherical caps in $\S(a) \cap \S(a')$.
  Equivalently, for a fixed spherical cap $S\in \S$ with direction vector
  $\overrightarrow{aa'}/\dist(a,a')$, the set $W_{a,a'}$ is the locus of
  directions $\vec{v}\in \mathbb{S}^2$ corresponding to the line segments
  of length $\dist(a,a')$ contained in $S$; see Fig.~\ref{fig:locus3}.
Note that $\dist(a,a')  \geq 1$, since the pairwise  minimum distance is equal to $1$;
  and if $\dist(a,a')>2r$, then $W_{a,a'}=\emptyset$, of measure 0.
We therefore have $1 \leq \dist(a,a') \leq 2r$. By Lemma~\ref{lem:locus},
$W_{a,a'}$ is contained in a spherical ring of polar angle $\theta$, where
$\theta/2= \arcsin \frac{h}{\dist(a,a')}$, and so its measure in $\mathbb{S}^2$
is $\O(\arcsin(h/\dist(a,a'))=\O(h/\dist(a,a'))$.
 Since the spherical caps from $\S(a) \cap \S(a')$ are in $\S(a)$, we have
 \begin{equation}\label{eq:c5}
 \mu(\S(a) \cap \S(a')) \leq c_5 \, \frac{h}{\dist(a,a')},
 \end{equation}
 for some constant $c_5>0$.
 Recall that $A(j)=\{a'\in A: j \leq \dist(a,a') \leq j+1\}$. Hence
  \begin{align*}
 \sum_{a' \in A} \mu(\S(a) \cap \S(a')) &= \mu(\S(a)) + \sum_{j=1}^{\lfloor 2r\rfloor} \sum_{a' \in A(j)} \mu(\S(a) \cap \S(a')) \\
 &\leq c_3 + \sum_{j=1}^{\lfloor 2r\rfloor} \sum_{a' \in A(j)} c_5 h j^{-1}
  \leq c_3 + \sum_{j=1}^{\lfloor 2r\rfloor} c_1 j^2 c_5 h j^{-1} \\
 &\leq c_3 + c_1 c_5 h \sum_{j=1}^{\lfloor 2r\rfloor} j
  \leq c_3 + 3 c_1 c_5 r^2 h
  \leq c_3 +  6 c_1 c_5 \alpha^3
  = c_4,
  \end{align*}
for some  constant $c_4>0$. Here we used part~(i), Inequality~\eqref{eq:c5}, Lemma~\ref{lem:shell},
and $r^2 h < 2 \alpha^3$.
\end{proof}

\smallskip\noindent\textbf{Proof of Lemma~\ref{lem:key}.}
We wish to bound $\sum_{i=1}^n m_i$ from below.
Consider the sums $\sum_{i=1}^n m_i$ and $\sum_{i=1}^n i^2 m_i$
for which we apply the Cauchy-Schwarz inequality in the following form:
\[ \left(\sum_{i=1}^n m_i\right) \left(\sum_{i=1}^n i^2 m_i\right) \geq \left( \sum_{i=1}^n i m_i \right)^2. \]
We can bound $\sum_{i=1}^n i m_i$ from below and $\sum_{i=1}^n i^2 m_i$ from above
as follows:
\begin{align*}
  \sum_{i=1}^n i m_i &= \sum_{a \in A} \mu(\S(a)) \geq  \sum_{a \in A} c_2 = c_2 \, n,
& \text{[by Lemma~\ref{lem:mu}(i)] } \\
 \sum_{i=1}^n i^2 m_i &= \sum_{a \in A} \sum_{a' \in A} \mu(\S(a) \cap \S(a')) \leq
 \sum_{a \in A} c_4 =  c_4 \, n.  & \text{[by Lemma~\ref{lem:mu}(ii)] }
\end{align*}

Applying these estimates yields the desired lower bound
\[  \sum_{i=1}^n m_i  \geq \frac{\left(\sum_{i=1}^n i m_i \right)^2}{\sum_{i=1}^n i^2 m_i}
\geq \frac{(c_2 \, n)^2}{c_4 \, n} = \frac{c_2^2}{c_4} \, n. \]
Consequently,
\[ \frac{\mu(\{S \in \S \colon S \cap A \neq \emptyset\})}{\mu(\S)}=
\frac{\sum_{i=1}^n m_i}{\mu(\S)} \geq
\frac{(c_2^2/c_4) n}{\frac{128}{3} \pi^3 \alpha^3 n} = c>0, \]
for some positive constant $c$.
\qed

\smallskip\noindent\textbf{Proof of Theorem~\ref{thm:lower}.}
We randomly place a ball $C$ congruent to $B$ inside $B'$.
Specifically, recall that $B$, $B'$, and $B''$ are concentric balls of radii $R$, $2R$, and $3R$,
respectively, where $A \subset B \subset B' \subset B''$.
We construct a random congruence as follows:
Let $\varrho$ be rotation in $\mathrm{SO}(3)$, the group of rotation in $\RR^3$, chosen uniformly at random;
and let $\tau$ be a translation that maps $O$ to point in $B'$ chosen uniformly at random.
Put $\Phi:=\tau\circ \varrho$, and $C:=\Phi(B)$.
Since the center of $C$ is in $B'$, then $C\subset B''$.
For any spherical cap $S_0\in \P$ of direction $\vec{v}\in \mathbb{S}^2$, the
direction $\varrho(\vec{v})$ is distributed uniformly on $\mathbb{S}^2$~\cite{Mil65}.
Consequently, the probability distribution of the indicator variable $I[\Phi(S_0)=S]$
over $S\in \S$ is a scalar multiple of $\mu$.
By Lemma~\ref{lem:mu}, $\Phi(S)$ contains some point in $A$ with a probability at least
\[ \frac{\mu(\{S \in \S \colon S \cap A \neq \emptyset\})}{\mu(\S)} \geq c, \]
for some constant $c>0$. By linearity of expectation and~\eqref{eq:P},
the expected number of nonempty spherical caps in
$\Phi(\P)$ is at least $\frac12\, c n^{1/2}$.
Setting $\beta=\frac{c}{2}$ completes the proof of Theorem~\ref{thm:lower}.
\qed

\smallskip\noindent\textbf{Proof of Theorem~\ref{thm:lower-general} (sketch).}
Set $R=\alpha n^\tau$, $h=\alpha n^{-\tau/2}$, $r=(R^2-(R-h)^2)^{1/2}=\Theta(\alpha n^{\tau/4})$,
and proceed as in the proof of Theorem~\ref{thm:lower}. With this setting we have $r^2 h =\Theta(\alpha^3)$.
The resulting lower bound is of the form
\[ \Omega \left( n^{1-3\tau} \, \frac{R^2}{r^2} \right) =  \Omega \left( n^{1-3\tau/2} \right), \]
as required.
\qed

\section{Upper bound: Proof of Theorem~\ref{thm:upper}} \label{sec:upper}

In this section, we describe and analyze a $3$-dimensional construction, similar to the Valtr and the Conlon-Lim grids.
It suffices to prove Theorem~\ref{thm:upper} for every $n$ of the form $n=2^{3k}$, where $k \in \NN$.
Our point set is a suitable perturbation of the $3$-dimensional Cartesian grid $G=\{0,1,\ldots,2^k-1\}^3$,
where each point lies within a ball of radius $\O(\eps)>0$ centered at an integer point in $G$.
Here $\eps< 10^{-2} 2^{-k}< 10^{-2}$ is a sufficiently small positive real that depends on $n$.
Recall that for a nonnegative integer $N=\sum_{i \geq 0} N_i 2^i$, where $N_i \in \{0,1\}$, and for $\eps\in (0,\frac12)$,
we have defined $(N)_\eps = \sum_{i \geq 0} N_i \eps^{i+1}$. Note that $0< (N)_\eps < 2\eps$.

We give an explicit formula for the perturbation $\Phi:G\rightarrow \RR^3$
in terms of $\eps>0$. Let $\mathbf{e}_1$, $\mathbf{e}_2$, and $\mathbf{e}_3$
denote the three standard basis vectors in $\RR^3$. For a grid point $p\in G$, let $p=(p_1,p_2,p_3)$
denote the three coordinates of $p$. For all $i,j\in \{1,2,3\}$, let
$\varphi_{i,j}:G\rightarrow \RR^3$, $\varphi_{i,j}(p)=(p_i)_\eps\cdot \mathbf{e}_j$.
 For every $p\in \ZZ^3$, let $\mathbf{u}(p)\in \RR^3$ be a random unit vector.
We can now define the perturbation $\Phi:G\rightarrow \RR^3$ as
\begin{align}
\hspace{-8mm}
\Phi(p)
  &=p+\left(\sum_{i\neq j} \eps^{(3i+j-5)k}\cdot \varphi_{i,j}(p)\right)+ \eps^{7k}\cdot \mathbf{u}(p)\label{eq:Phi}\\
  &= \left(p_1+\eps^{2k}(p_2)_\eps + \eps^{5k}(p_3)_\eps,
             p_2+(p_1)_\eps+\eps^{6k} (p_3)_\eps,
             p_3+\eps^{k}(p_1)_\eps+ \eps^{4k}(p_3)_\eps\right) +\eps^{7k} \mathbf{u}(p), \nonumber
\end{align}
and let $A=\Phi(G)$. The last term, $\eps^{7k}\cdot \mathbf{u}(p)$, ensures that $A$ is in strongly general position.
It is convenient to think of $\Phi$ as a concatenation of seven successive perturbations, corresponding to the terms
in~\eqref{eq:Phi}. Terms with different powers of $\eps$ ensure that each successive perturbation is negligible
with respect to previous perturbations if $\eps>0$ is sufficiently small. We introduce notation for the result of
the first perturbation: let $\Phi_1:G\rightarrow \RR^3$, $\Phi_1(p)=p+\varphi_{1,2}(p) = \left(p_1,p_2+(p_1)_\eps,p_3\right)$.

We next analyze this construction and show that it contains no large convex subsets as quantified in
Theorem~\ref{thm:upper}. Let $B\subset A$ be a set in convex position,
and let $C=\Phi^{-1}(B)$, \ie, $C\subset G$ with $B=\Phi(C)$.
If $\eps>0$ is sufficiently small, then for every vertex $p$ of $\conv(C)$,
the point $\Phi(p)$ is a vertex of $\conv(B)$.
However, for every vertex $q$ of $\conv(B)$, the point $\Phi^{-1}(q)$ is either a vertex of $\conv(C)$
or lies in the interior of an edge or a face of $\conv(C)$. We have
\[ \conv(C)\subset \conv(G), \text{ and } \vol(\conv(C))\leq \vol(\conv(G))=n. \]
By Theorem~\ref{thm:Andrews}, $\conv(C)$ has $\O(n^{(3-1)/(3+1)})=\O(n^{1/2})$ vertices.
By Euler's polyhedral formula, $\conv(C)$ has $\O(n^{1/2})$ edges and facets.
We will show that if an edge of $\conv(C)$ contains $m$ lattice points in $G$, then
only $\O(\log^2 m)$ of these points are in $C$ (Lemma~\ref{lem:edge});
and if a face $F$ of $\conv(C)$ contains $m$ lattice points in $G$,
then $\O(m^{1/3})$ of these points are in $C$ (Lemma~\ref{lem:face}).

We use three lemmas (Lemmas~\ref{lem:sub-add}, \ref{lem:edge}, and~\ref{lem:face}).
Lemma~\ref{lem:sub-add} describes the subadditivity of the function $g(\cdot)$
(defined at the end of Section~\ref{sec:intro}). Its easy proof is left to the reader.

\begin{lemma} \label{lem:sub-add}
Let $S= \bigcup_{i=1}^k S_i$ be an arbitrary partition of $S$. Then $g(S) \leq \sum_{i=1}^k g(S_i)$.
\end{lemma}

Lemma~\ref{lem:edge} below states that in a perturbation of a set of collinear points in the grid $G$,
there are only $\O(\log^2 n)$ points in convex position. A possible proof for Lemma~\ref{lem:edge} would establish
that the perturbation of collinear points in $G$ is a Horton set in $\RR^3$, and then apply the following result
due to K{\'a}rolyi and Valtr~\cite{KV03}.

\begin{lemma} \label{lem:horton} {\rm (K{\'a}rolyi and Valtr~\cite{KV03})}
  Let $S \in \RR^d$ be a Horton set of size $n$. Then $g_d(S) = \O(\log^{d-1} n)$, where the
  constant hidden in the $\O$ notation depends only on $d$.
\end{lemma}

Instead, we give a direct proof for Lemma~\ref{lem:edge} and then generalize it to handle a perturbation
of a set of coplanar points in Lemma~\ref{lem:face} below.

\begin{lemma}\label{lem:edge}
Let $S \subset G$ be a set of collinear points. Then $\conv(\Phi(S))$ has $\O(\log^2 m)$ vertices,
where $m= |G \cap \conv(S)|$.
\end{lemma}
\begin{proof}
We may assume that $|S|\geq 2$. Let $L$ be the line spanned by $S$, and let $m=|G\cap \conv(S)|$.
Consider the first coordinate axis ($x$-, $y$-, or $z$-axis) that is not orthogonal to $L$.
We give a detailed proof for the case that the $x$-axis is not orthogonal to $L$.
The other two cases are analogous (and are omitted):
If $L$ is orthogonal to the $x$-axis, then the first two iterations of the perturbation
(which depend on the first coordinate) translate all points $S$ uniformly,
hence they have no impact on the convex hull of $\Phi(S)$. Similarly,
if $L$ is orthogonal to both $x$- and $y$-axes, then we can ignore the components of the perturbation
that depend on the first and second coordinates (and use the third coordinate instead).

Assume that $L$ is not orthogonal to the $x$-axis.
Label the points in $S$ as $S=\{s_1,\ldots , s_t\}$ sorted by increasing $x$-coordinates.
Let the binary representation of the $x$-coordinate of $s\in S$ be $x(s)=\sum_{j=0}^{k-1} x_j(s) 2^j$.
We recursively define the sets
\begin{equation}\label{eq:nest}
    S_1\supset S_2\supset \ldots \supset S_b,
\end{equation}
for some suitable $b\leq \O(\log m)$ as follows; see Fig.~\ref{fig:edge}.
Let $S_1=S$. Given a set $S_a\subset S$, for $a\geq 1$, with $|S_a|\geq 2$, we define $S_{a+1}$ as follows.
Let $j(a)\geq 0$ be the smallest integer such that $\{x_{j(a)}(s): s\in S_a\}=\{0,1\}$;
and let $S_{a+1}=\{s\in S_a: x_{j(a)}(s)=1\}$.
Importantly, this implies that the translation vector
$\varphi_{1,2}(s)$ has a term $\eps^{j(a)}\mathbf{e}_2$ for all $s\in S_{a+1}$;
but this term is missing for all $s\in S_a\setminus S_{a+1}$.
Note that for all $s\in S_a$, the last (\ie, the least significant)
$j(a)$ bits in the binary expansion of $x(s)$ are the same.
Consequently, $j(a+1)>j(a)$ for all $a\geq 1$; hence $a\leq \O(\log m)$,
and the recursion terminates with $b\leq \O(\log m)$, as claimed.
Note also that $|S_b|=1$ by definition.

\begin{figure}[htbp]
\centering
\includegraphics[width=\textwidth]{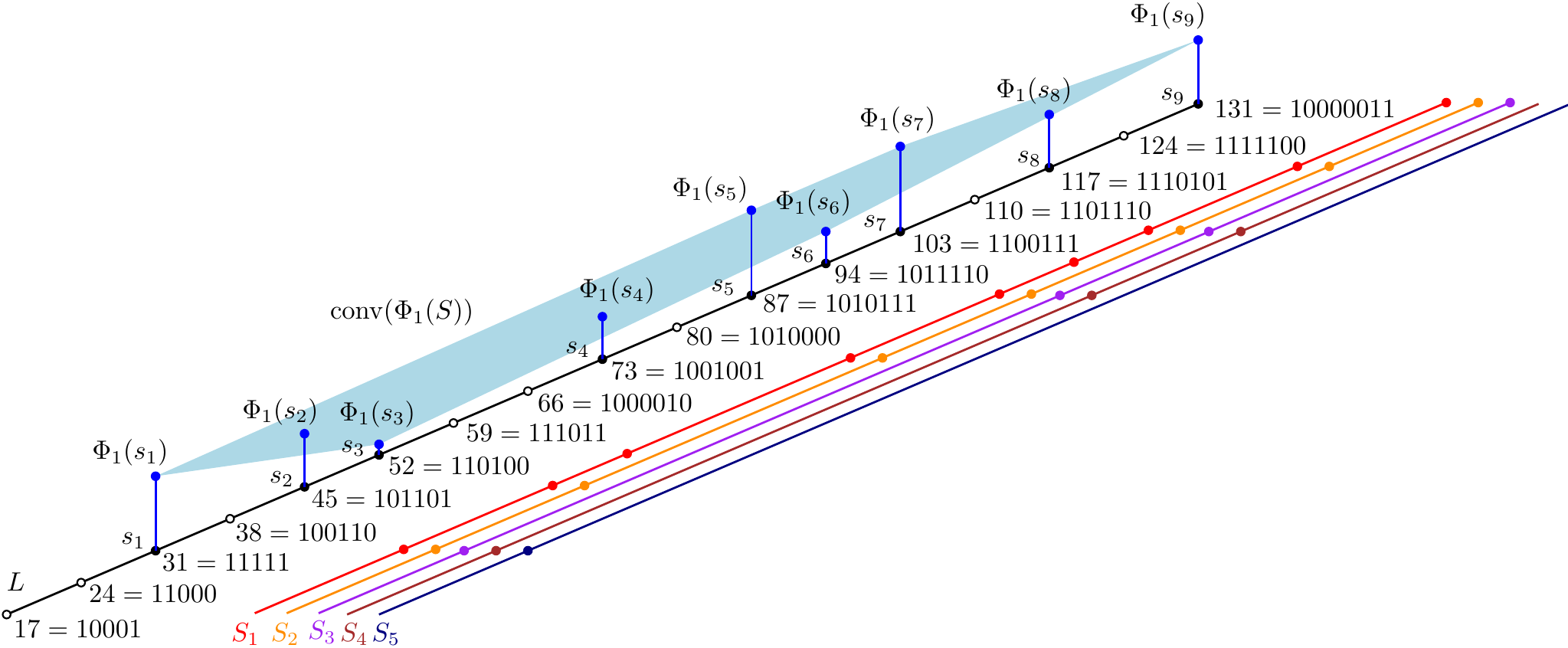}
\caption{Perturbation for a set of collinear grid points. Left: Solid dots indicate $S=\{s_1,\ldots , s_9\}$;
  all other grid points are marked by empty dots.
  The binary expansion of the $x$-coordinates of all grid points.
Right: The elements of $S_1,\ldots, S_5$, where  $j(1)=0$, $j(2)= 1$, $j(3)=2$, and $j(4)=3$.}
\label{fig:edge}
\end{figure}

We show that $\conv(\Phi(S))$ has $\O(\log^2 m)$ vertices in two steps:
First we consider the convex hull of a projection of $\Phi(S)$ to a plane,
and then extend the argument to 3-space.

\smallskip\noindent\textbf{Convex hull of the orthogonal projection to the $xy$-plane.}
We consider the impact of the first step of the perturbation, $\Phi_1(S)$.
Let $\Phi^{xy}(S)$ and $\Phi_1^{xy}(S)$, resp., be the orthogonal projection of $\Phi(S)$ and
$\Phi_1(S)$ to the $xy$-plane.

\begin{claim}\label{cl:1}
The convex polygon $\conv(\Phi^{xy}(S))$ has $\O(\log m)$ vertices.
\end{claim}
If a point $\Phi^{xy}(p)$ is a vertex of $\conv(\Phi^{xy}(S))$, then
$\Phi_1^{xy}(p)$ is on the boundary of $\conv(\Phi_1^{xy}(S))$,
since subsequent perturbations are negligible.
Note that both $\Phi_1^{xy}(s_1)$ and $\Phi_1^{xy}(s_t)$ are vertices of $\conv(\Phi_1^{xy}(S))$.
These two points decompose the boundary of $\conv(\Phi_1^{xy}(S))$ into two Jordan arcs:
an \emph{upper arc} and a \emph{lower arc} in the $xy$-plane.
It suffices to show that each arc contains $\O(\log n)$ points of $\Phi_1^{xy}(S)$.
Without loss of generality, consider the upper arc.

\begin{claim}\label{cl:2}
Let $s_\ell \in S$.
If $\Phi_1^{xy}(s_\ell)$ lies in the upper arc of $\conv(\Phi_1^{xy}(S))$,
then $s_\ell$ is the first or last point in $S_a$ for some $a\in \{1,\ldots , b\}$.
\end{claim}
We prove the contrapositive of Claim~\ref{cl:2}.
Suppose that $s_\ell$ is neither the first nor the last point in
$S_a$ for any $a\in \{1,\ldots , b\}$.
Due to \eqref{eq:nest} and $|S_b|=1$, there exists an $a\in \{1,\ldots , b-1\}$
such that $s_\ell\in S_a$ but $s_\ell\notin S_{a+1}$.
Denote the first and last points in $S_a$ and $S_{a+1}$, resp.,
by $s_{\sigma(1)},s_{\sigma(2)}$ and $s_{\tau(1)},s_{\tau(2)}$.
Then $\conv(S_a)$ and $\conv(S_{a+1})$, resp., are the line segments
$s_{\sigma(1)}s_{\sigma(2)}$ and $s_{\tau(1)}s_{\tau(2)}$; see Fig.~\ref{fig:heights}.
Note that $\sigma(1) \leq \tau(1) \leq \tau(2) \leq \sigma(2)$ and $\sigma(1) < \sigma(2)$.
Then we have $\sigma(1)<\ell<\tau(2)$ or $\tau(1)<\ell<\sigma(2)$.
Assume w.l.o.g.\ that $\sigma(1)<\ell<\tau(2)$.

We show that the point $\Phi_1^{xy}(s_\ell)$ lies below the line segment
$\Phi_1^{xy}(s_{\sigma(1)}) \Phi_1^{xy}(s_{\tau(2)})$, and so it cannot be on the upper arc of $\conv(\Phi_1^{xy}(S))$.
Since the points $s_{\sigma(1)}, s_\ell,s_{\tau(2)}\in L$ are collinear,
and $\Phi_1(p)=p+\varphi_{1,2}(p)$,
it is enough to compare the perturbations incurred by $\varphi_{1,2}$, which depend only on the $x$-coordinates.
By the choice of $a$, we have $s_{\sigma(1)}, s_\ell,s_{\tau(2)}\in S_a$.
This means that in the binary expansion of their $x$-coordinates,
the last $j(a)$ bits are the same. Regarding the bit $x_{j(a)}(.)$,
we know that $x_{j(a)}(s_{\ell})=0$ and $x_{j(a)}(s_{\tau(2)})=1$.
However, $s_{\sigma(1)}$ may or may not be in $S_{a+1}$, and so we do not know $x_{j(a)}(s_{\sigma(1)})$.
Consequently, by using the inequality $(N)_\eps < 2\eps$ in~\eqref{eq:z2}, we have
 \begin{align}
 \|\varphi_{1,2}(s_{\sigma(1)})\| &
    \geq \left(\sum_{j=0}^{j(a)-1} x_j(s_{\ell})\cdot \eps^{j+1}\right) =:X \label{eq:z1}\\
 \|\varphi_{1,2}(s_{\ell})\| &
    < \left(\sum_{j=0}^{j(a)-1} x_j(s_{\ell})\cdot \eps^{j+1}\right) + 2\eps^{j(a)+2}
    = X+2\eps^{j(a)+2}\label{eq:z2}\\
 \|\varphi_{1,2}(s_{\tau(2)})\| &
    \geq \left(\sum_{j=0}^{j(a)-1} x_j(s_{\ell})\cdot \eps^{j+1}\right) + \eps^{j(a)+1}
    = X+\eps^{j(a)+1}. \label{eq:z3}
\end{align}
Since $s_{\sigma(1)}$, $s_\ell$, and $s_{\tau(2)}$
are collinear grid points, we can express $s_\ell$ as a convex combination:
\begin{align}\label{eq:cx}
    s_\ell &= \alpha\cdot s_{\sigma(1)} + (1-\alpha)\cdot s_{\tau(2)}
\end{align}
for some coefficient $\frac{1}{n^{1/3}}\leq \alpha \leq 1-\frac{1}{n^{1/3}}$.
Denote by $q$ the point in the line segment $\Phi_1^{xy}(s_{\sigma(1)})\Phi_1^{xy}(s_{\tau(2)})$
above $s_\ell$. Then, substituting \eqref{eq:z1}--\eqref{eq:z3} into \eqref{eq:cx}, we obtain
\begin{align}
    \|q-s_\ell\|
    & = \alpha\cdot \|\varphi_{1,2}(s_{\sigma(1)})\| + (1-\alpha)\cdot \|\varphi_{1,2}(s_{\tau(2)})\| \label{eq:height}\\
    &\geq \left(\sum_{j=0}^{j(a)-1} x_j(s_\ell)\cdot \eps^{j+1}\right) + (1-\alpha)\cdot \eps^{j(a)+1}\nonumber\\
    &\geq X +\frac{\eps^{j(a)+1}}{n^{1/3}} > X+2\eps^{j(a)+2}
     > \|\varphi_{1,2}(s_{\ell})\|, \nonumber
\end{align}
if $\eps< \frac{1}{2n^{1/3}}$.  This confirms that $\Phi_1^{xy}(s_\ell)$ lies below the line segment
$\Phi_1^{xy}(s_{\sigma(1)})\Phi_1^{xy}(s_{\tau(2)})$; and completes the proof of Claim~\ref{cl:2}.

Since $b=\O(\log m)$, there are at most $2b=\O(\log m)$ points that are first or last in $S_1,\ldots ,S_b$.
Consequently, Claim~\ref{cl:2} implies that the upper arc of $\conv(\Phi^{xy}(S))$ has $\O(\log m)$ vertices.
Similarly, one can show that the lower arc of $\conv(\Phi^{xy}(S))$ has $\O(\log m)$ vertices
(the key difference for handling the lower arc is that in the recursive definition of the
sets $S_1\supset S_2\supset \ldots \supset S_b$, we would put $S_{a+1}=\{s\in S_a: x_{j(a)}(s)=0\}$).
Overall, $\conv(\Phi^{xy}(S))$ has $\O(\log m)$ vertices; completing the proof of Claim~\ref{cl:1}.

\begin{figure}[htbp]
\centering
\includegraphics[width=.6\textwidth]{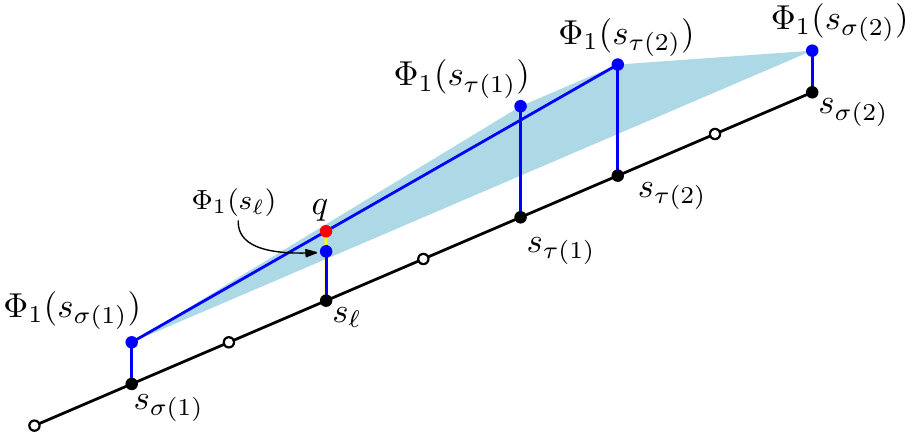}
\caption{The points $s_\ell$, $s_{\sigma(1)}$, $s_{\sigma(2)}$, $s_{\tau(2)}$, and $s_{\tau(2)}$; and their perturbations.}
\label{fig:heights}
\end{figure}
%

\smallskip\noindent\textbf{Convex hull in 3-space.}
Consider $\conv(\Phi(S))$. The orthogonal projection of $\conv(\Phi(S))$
to the $xy$-plane is $\conv(\Phi^{xy}(S))$.
The boundary of $\conv(\Phi^{xy}(S))$ is the projection of a closed
curve $\gamma$ on the boundary of $\conv(\Phi(S))$.
We have shown that the projection of $\gamma$,
hence $\gamma$ itself, has $\O(\log n)$ vertices.
The Jordan curve $\gamma$ partitions the boundary of $\conv(\Phi(S))$
into two components: An \emph{upper surface} and a \emph{lower surface}.
It suffices to show that each has $\O(\log^2 m)$ vertices.

\begin{claim}\label{cl:3}
Let $s_\ell\in S$. If $\Phi(s_\ell)$ is a vertex of the upper surface of $\conv(\Phi(S))$,
then $\Phi^{xy}(s_\ell)$ is a vertex of $\conv(\Phi^{xy}(S_a))$ for some $a\in \{1,2,\ldots , b\}$.
\end{claim}
We prove the contrapositive of Claim~\ref{cl:3}.
Suppose that $s_\ell\in S$ but $\Phi(s_\ell)$ is not a vertex of $\conv(\Phi(S_a))$
for any $a\in \{1,\ldots , b\}$. Note that \eqref{eq:nest} implies
\[
    \conv(\Phi^{xy}(S))
    =\conv(\Phi^{xy}(S_1))
    \supset \conv(\Phi^{xy}(S_2))
    \supset \ldots
    \supset \conv(\Phi^{xy}(S_b)).
\]
There exists some $a\in \{1,\ldots , b\}$ such that
$s_\ell\in S_a$ but $s_\ell\notin S_{a+1}$.
First triangulate the convex polygon $\conv(\Phi^{xy}(S_{a+1}))$; and then triangulate the nonconvex polygon
$\conv(\Phi^{xy}(S_a))\setminus \conv(\Phi^{xy}(S_{a+1}))$
such that each triangle is spanned by some vertices of $\conv(\Phi^{xy}(S_a))$ and
some vertices of $\conv(\Phi^{xy}(S_{a+1}))$.
The point $\Phi^{xy}(s_\ell)$ lies in some triangle $\Phi^{xy}(p)\Phi^{xy}(q)\Phi^{xy}(r)$
in which at least one corner is a vertex of $\conv(\Phi^{xy}(S_{a+1}))$.
Considering the perturbation $\Phi_1(s)=s+\varphi_{1,2}(s)$,
the triangle $\Phi(p)\Phi(q)\Phi(r)$ lies above $\Phi(s_\ell)$ if
$\eps>0$ is sufficiently small, analogously to \eqref{eq:z1}--\eqref{eq:z3} and \eqref{eq:cx},
except that $s_\ell$ is now the convex combination of three grid points $p$, $q$, and $r$.
Consequently, $\Phi(s_\ell)$ cannot be on the upper surface.
This proves Claim~\ref{cl:3}.

By Claim~\ref{cl:2}, $\conv(\Phi^{xy}(S_a))$ has $\O(\log m)$ vertices for all $a\in \{1,2,\ldots , b\}$.
Summation over all $a=1,\ldots, b$ shows that there are $\sum_{a=1}^b \O(\log m) = \O(\log^2 m)$ vertices
on the upper surface of $\conv(\Phi(S)$.
Analogously, the lower surface of $\conv(\Phi(S))$ also has $\O(\log^2 m)$ vertices.
Overall, $\conv(\Phi(S))$ has $\O(\log^2 m)$ vertices.
This completes the proof of Lemma~\ref{lem:edge}.
\end{proof}

\begin{lemma}\label{lem:face}
 Let $S \subset G$ be a set of coplanar points. Then $\conv(\Phi(S))$ has $\O(m^{1/3})$ vertices,
where $m= |G \cap \conv(S)|$.
\end{lemma}
\begin{proof}
We proceed similarly to the proof of Lemma~\ref{lem:edge}.
Consider the first coordinate axis (the $x$- or the $y$-axis) that is not orthogonal to $F$, and
another coordinate axis that is not parallel to $F$.
Assume that $F$ is not orthogonal to the $x$-axis (or else we would use the $y$-axis);
and it is not parallel to the $y$-axis (or else we would use the $z$-axis).
Label the points in $S\cap F$ as $\{s_1,\ldots , s_t\}$ sorted by increasing $x$-coordinates
(ties are broken arbitrarily).
Let the binary representation of the $x$-coordinate of $s\in S$ be $x(s)=\sum_{j=0}^{k-1} x_j(s) 2^j$.
We recursively define the sets
\begin{equation}\label{eq:nest+}
    S_1\supset S_2\supset \ldots \supset S_b,
\end{equation}
for some suitable $b\leq \O(\log m)$ the same way as in the proof of Lemma~\ref{lem:edge}.
Let $S_1=S$. Given a set $S_a\subset S$, for $a\geq 1$, where the points in $S_a$ do not all have the same
$x$-coordinates, we define $S_{a+1}$ as follows. Let $j(a)\geq 1$ be the smallest integer such that
$\{x_{j(a)}(s): s\in S_a\}=\{0,1\}$; and let $S_{a+1}=\{s\in S_a: x_{j(a)}(s)=1\}$.

Denote by $M=\ZZ^3\cap F$ the set of grid points in $F$. Based on the recursion above, we define a sequence of nested sets
\begin{equation}\label{eq:grids}
    M_1\supset M_2\supset \ldots \supset M_b
\end{equation}
as follows. Let $M_1=M$, and $M_{a+1}=\{s\in M_a: x_{j(a)}(s)=1\}$ for $a\in \{1,\ldots , b-1\}$.
Then $|M_a| = \O(m/2^a)$ for all $a\in \{1,\ldots , b\}$.

\begin{figure}[htbp]
\centering
\includegraphics[width=0.55\textwidth]{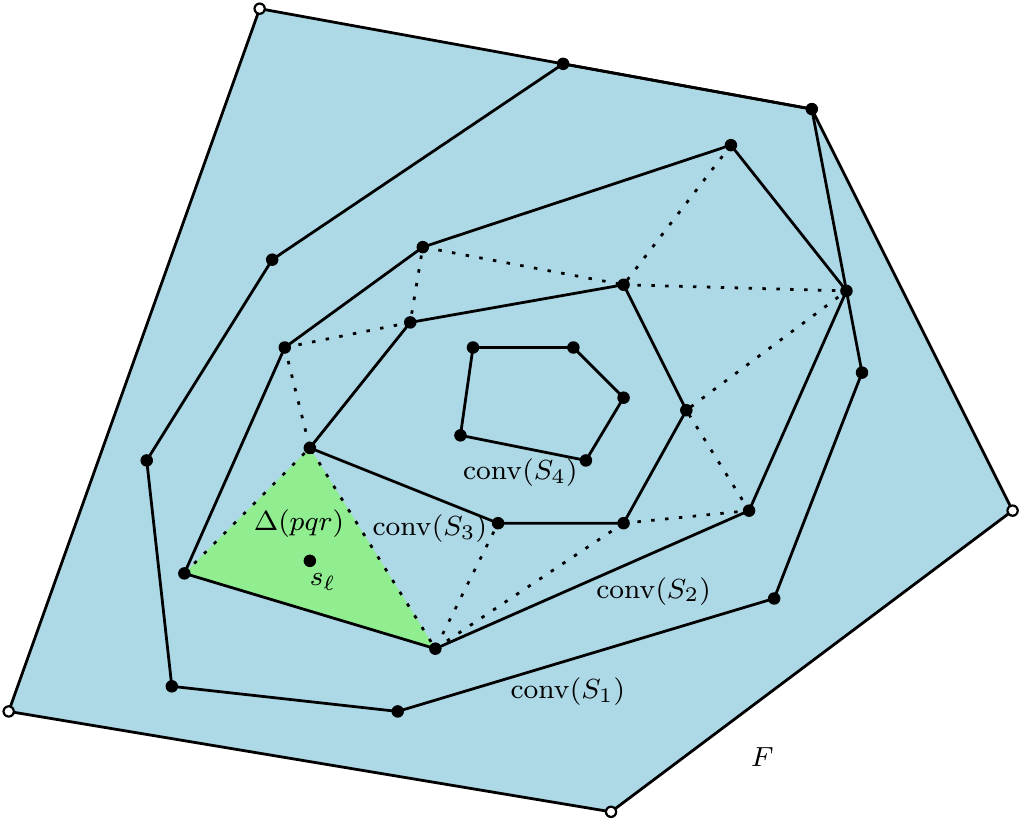}
\caption{The nested sequence $F\supset \conv(S_1)\supset \ldots \supset \conv(S_4)$.
  The point $s_\ell$ lies in $\conv(S_2)\setminus \conv(S_3)$.}
\label{fig:face}
\end{figure}

We interpret the above-below relationship with respect to the $y$-axis.
Assume w.l.o.g.\ that $\conv(S)$ lies below the face $F$.

\begin{claim}\label{cl:4}
Let $s_\ell\in S\cap F$. If $\Phi(s_\ell)$ is a vertex of $\conv(\Phi(S\cap F))$,
then $s_\ell$ lies on the boundary of $\conv(S_a)$ for some $a\in \{1,\ldots , b\}$.
\end{claim}
We prove the contrapositive of Claim~\ref{cl:4}.
Suppose that $s_\ell$ is not on the boundary of $\conv(S_a)$ for any $a\in \{1,\ldots , b\}$.
Due to \eqref{eq:nest}, there exists some $a\in \{1,\ldots , b\}$ such that
$s_\ell\in S_a$ but $s_\ell\notin S_{a+1}$; see Fig.~\ref{fig:face}.
First triangulate the convex polygon $\conv(S_{a+1})$;
and then triangulate the nonconvex polygon
$\conv(S_a)\setminus \conv(S_{a+1})$ such that each triangle
is spanned by some vertices of $\conv(S_a)$ and some vertices of $\conv(S_{a+1})$.
Then point $s_\ell$ lies in some triangle $pqr$, and at least one corner is a vertex of $\conv(S_{a+1})$.
Considering the perturbation $\varphi_{1,2}$,
the triangle $\Phi(p)\Phi(q)\Phi(r)$ lies above $\Phi(s_\ell)$ if $\eps>0$ is sufficiently small.
Consequently, $\Phi(s_\ell)$ cannot be a vertex of $\conv(\Phi(S\cap F))$.
This completes the proof of Claim~\ref{cl:4}.

For every $a\in \{1,\ldots , b\}$, the convex hull $\conv(S_a\cap M_a)$ has $\O(m_a^{1/3})$ vertices and edges,
where $m_a=|M_a|\leq \O(m/2^a)$, by Theorem~\ref{thm:Andrews}.
Further, by Lemma~\ref{lem:new2}, it has $\O((m_a/t^2)^{1/3})$ edges that contain more than $t$ points in $M_a$
for any $t\geq 1$. By Lemma~\ref{lem:edge}, if an edge $e$ of $\conv(S_a\cap M_a)$ contains more than $t$ grid points,
then $\conv(\Phi(S\cap e))$ contains $\O(\log^2 t)$ vertices of $B$.
Let $E_j$ be the set of edges $e$ with $|e\cap \ZZ^3|\in (2^{j-1}, 2^j]$;
and $E=\bigcup_{j\in \NN} E_j$. Summation over all edges of $S_a\cap F$ yields
\begin{align*}
\sum_{e\in E} \O(\log^2 |e\cap \ZZ^3|)
 &=\sum_{j\in \NN} |E_j|\cdot \O(\log^2 2^j)\\
 &=\sum_{j\in \NN}  \O\left( (m_a/2^{2j})^{1/3} \cdot j^2\right)\\
 &=\O(m_a^{1/3})\, \sum_{j\in \NN} \frac{j^2}{(2^{2/3})^j} = \O(m_a^{1/3}).
\end{align*}
Finally, summation for all $a\in \{1,\ldots ,b\}$ yields
\[
  \sum_{a=1}^b \O(m_a^{1/3})
= \sum_{a=1}^b  \O\left((m/2^a)^{1/3} \right)
= \O\left(m^{1/3}\, \sum_{a\geq 1} 2^{-a/3}\right)
= \O(m^{1/3}), \]
as claimed.
\end{proof}

\begin{lemma}\label{lem:summation}
Let $S\subset G$ and $P=\conv(S)$.
Suppose that $P$ has $f$ faces (of any dimension),
which contain $m_1,\ldots, m_f$ lattice points in their interior.
Then $\sum_{i=1}^f m_i^{1/3} =\O(n^{1/2})$.
\end{lemma}
\begin{proof}
  The surface area of the lattice polytope $P=\conv(S)$ is bounded above by that of the cube $\conv(G)$,
  which is $6n^{2/3}$. Let $\mathcal{F}$ be the set of all facets of $P$. Then summation of the area over all facets yields
  an upper bound $\sum_{F\in \mathcal{F}}\O(\area(F)) = \O(n^{2/3})$.

  For every face $F\in \mathcal{F}$, let $\mathbf{n}_F=(a,b,c)$ be the integer normal vector of the plane spanned by $F$,
  where $a,b,c\in \ZZ$, and $\mathrm{gcd}(a,b,c)=1$. It is known that $F$ contains $\O(\area(F)/\|\mathbf{n}\|_2)$
  lattice points~\cite{And61}.

  For every integer $j\in \NN$, let $\mathcal{F}_j$ be the set of facets $F\in \mathcal{F}$
  such that $\|\mathbf{n}_F\|_2\in (2^{j-1}, 2^j]$.
  In particular, a face $F\in \mathcal{F}_j$ contains $\O(\area(F)/2^j)$ lattice points.
  For an integer $N\in \NN$, let $r_3(N)$ denote the number of representations of $N$ as a sum of squares
  of three integers (where signs and the order of terms matters).
  It is known~\cite[Theorem~340]{HW79} that
\[ \sum_{N=1}^{M} r_3(N) = \frac43 \pi M^{3/2}+\O(M). \]
  In particular, for $M=2^{2j}$, we have $\sum_{N=1}^M r_3(N) =\Theta(2^{3j})$,
  and so $|\mathcal{F}_j| = \O(2^{3j})$.

The total area of all facets in $\mathcal{F}_j$ is $\O(n^{2/3})$ for each $j\in \NN$.
Thus the facets in $\mathcal{F}_j$ contain at most
$\sum_{F\in \mathcal{F}_j} \O(\area(F)/2^j) = \O(n^{2/3}/2^j)$ lattice points.
Jensen's inequality gives
\begin{align}\label{eq:Jensen}
 \sum_{F\in \mathcal{F}_j} m_i^{1/3}\log m_i
 &\leq |\mathcal{F}_j| \cdot \left(\frac{n^{2/3}/2^j}{|\mathcal{F}_j|}\right)^{1/3} \nonumber\\
 &= |\mathcal{F}_j|^{2/3} n^{2/9} 2^{-j/3}
 = \O(n^{2/9} 2^{(2/3)3j-j/3})
 = \O(n^{2/9} 2^{5j/3}).
\end{align}
Since the number of facets is $\O(n^{1/2})$, then $|\mathcal{F}_j| =\O(n^{1/2})$ for all $j\in \NN$;
and  \eqref{eq:Jensen} becomes:
\begin{align}\label{eq:Jensen2}
 \sum_{F\in \mathcal{F}_j} m_i^{1/3}
 &\leq |\mathcal{F}_j| \cdot \left(\frac{n^{2/3}/2^j}{|\mathcal{F}_j|}\right)^{1/3} \nonumber\\
 &\leq |\mathcal{F}_j|^{2/3} n^{2/9} 2^{-j/3}
 = \big( \O(n^{1/2}) \big)^{2/3} n^{2/9} 2^{-j/3}
 = \O(n^{5/9} 2^{-j/3}) .
\end{align}

The bounds in \eqref{eq:Jensen} and \eqref{eq:Jensen2} are equal, up to constant factors,
when $2^j=n^{1/6}$, or equivalently, $j=\frac16 \log n$.
Summation over all $j\in \NN$ yields
\begin{align*}
 \sum_{i=1}^f m_i^{1/3}
 &=\sum_{j\in \NN} \sum_{F\in \mathcal{F}_j} m_i^{1/3}\\
 &=\sum_{j=1}^{\log n^{1/6}} \sum_{F\in \mathcal{F}_j} m_i^{1/3}
    + \sum_{j> \log n^{1/6}} \sum_{F\in \mathcal{F}_j} m_i^{1/3}\\
 &\leq \sum_{j=1}^{\log n^{1/6}}  \O(n^{2/9} 2^{5j/3}) +  \sum_{j> \log n^{1/6}} \O(n^{5/9}/ 2^{j/3})\\
 &= \O(n^{2/9+5/18})+ \O(n^{5/9-1/18}) = \O(n^{1/2}).   \qedhere
\end{align*}
\end{proof}

We can now complete the proof of Theorem~\ref{thm:upper}.
Recall that $G=\{0,\ldots, 2^k-1\}^3$ is a section
of the integer lattice $\ZZ^3$ with $n=2^{3k}$ points.
We have $A=\Phi(G)$; and $C\subset G$ such that $B=\Phi(C)$ is in convex position.
As noted above, every point in $C$ lies on the boundary of
the lattice polytope $P=\conv(C)$.

If a facet $F$ of $P$ contains $m$ lattice points, then $F$ contains
$\O(m^{1/3})$ points of $C$ by Lemma~\ref{lem:face}.
Summation over all facets yields $\O(n^{1/2})$ by Lemma~\ref{lem:summation}.
Consequently, we have $|B|=|C| = \O(n^{1/2})$, as required.

\section{Approximation algorithm: Proof of Theorem~\ref{thm:approx}}

In this section we analyze the randomized algorithm described in Section~\ref{sec:lower} and show
that its approximation factor is $\Omega(n^{-1/6})$. We also make some small twists that allow for
an efficient implementation.
We are given an $n$-element set $A\subset \RR^3$ with $\Delta(A)\leq \O(n^{1/3})$
As in the proof of Theorem~\ref{thm:lower}, we may assume that $A \subset B$,
where $B$ is a ball of radius $R\leq \O(n^{1/3})$ centered at $O$; e.g., a smallest enclosing ball of $A$.
First, deterministically construct a (single) spherical packing $\P$ on
a sphere $C$ congruent to $B$ of prescribed minimum size $\Theta(n^{1/2})$;
($\P$ does not have to be maximal).
For example, slice the sphere by latitudes, and choose equally spaced spherical caps between
consecutive latitudes. Second, the randomized phase proceeds as follows.
Guess a center for a ball $C$ and apply a random rotation in $SO(3)$
(applying the same rotation to all caps in $\P$), and then test whether
a constant fraction of the caps are nonempty; if not, repeat the process.
The expected number of repetitions is bounded by a constant (the sum of a geometric series).

\smallskip\noindent\textbf{Approximation ratio.}
As in the proof of Theorem~\ref{thm:lower}, we may assume that $A \subset B$, where $B$ is a ball of radius
$R= \alpha n^{1/3}$ centered at $O$.
Let $A_{\opt} \subset A$ be a maximum-size subset in convex position, \ie, $\opt=|A_{\opt}|$.
Since $A_{\opt} \subset B$, we have
\[ \area(\partial(\conv(A_{\opt}))) \leq \area(\partial B)= 4\pi R^2 = \O(n^{2/3}). \]
Since $A$ is density-restricted, then $\opt=\O(n^{2/3})$.
(The above argument is the $3$-dimensional variant of~\cite[Theorem~3.2]{AKP89}.)
On the other hand, Theorem~\ref{thm:lower} yields
$\alg = \Omega(n^{1/2})$. Consequently, the approximation ratio is
\[ \frac{\alg}{\opt}= \Omega \left( \frac{n^{1/2}}{n^{2/3}} \right) = \Omega(n^{-1/6}), \]
as claimed.

\smallskip\noindent\textbf{Running time analysis.}
A smallest enclosing ball of $n$ points in $\RR^3$ can be computed in $\O(n)$ expected time~\cite{Wel91},
and a random rotation and translation in $\O(1)$ time~\cite{Sho92}.
The packing $\P$ can be constructed in $\O(n^{1/2})$ time (\ie, in time linear in $|\P|$).
We next consider the time complexity of $\Theta(n^{1/2})$ range-emptiness queries for spherical cap ranges.
After $\O(n)$ preprocessing, points outside the chosen ball are excluded from further consideration.
Then each range-emptiness query for a spherical cap is equivalent to (and answered by) a
\emph{halfspace emptiness query} determined by the plane containing the base of the cap.
(Here we take advantage of the fact that all spherical cap ranges pertain to the same ball.)
After $\O(n \log{n})$ expected preprocessing time,
such queries in $3$-space can be answered in $\O(\log{n})$ time per query using the algorithm
by Afshani and Chan~\cite{AC09}; but also by other algorithms, see~\cite{Aga17a,Aga17b}.
Consequently, the range-emptiness queries can be answered in  $\O(n^{1/2} \log{n})$ time.
Adding up the running times of the steps we have
$\O(n^{1/2}) + \O(n) + \O(n \log{n}) + \O(n^{1/2} \log{n}) =\O(n \log{n})$.
Overall, the randomized algorithm runs in $\O(n \log{n})$ expected time.
\qed

\subparagraph{Generalization to higher dimensions.}
The machinery developed here generalizes to $\mathbb{R}^d$.

\begin{theorem} \label{thm:approx-d}
Given any $n$-element point set in $\RR^d$ in general position, satisfying $\Delta(A) \leq \alpha n^{1/d}$
for a fixed $\alpha$, a $\Omega \left( n^{-\frac{(d-1)}{d(d+1)}} \right)$-factor approximation of the maximum size
convex subset of points can be computed by a randomized algorithm in $\O(n \log{n})$ expected time.
\end{theorem}

The proof of Theorem~\ref{thm:approx-d} is analogous to the proof of Theorem~\ref{thm:approx}.
The approximation ratio is
\[ \frac{\alg}{\opt}= \Omega \left( \frac{n^{(d-1)/(d+1)}}{n^{(d-1)/d}} \right) =
\Omega \left( n^{-\frac{(d-1)}{d(d+1)}} \right). \]
As the exponent tends to zero when $d \to \infty$, the approximation ratio improves with the dimension
(if $n$ is sufficiently large). As such, our algorithm enjoys the `blessing of dimensionality' rather than
the usual `curse of dimensionality'.

\section{Concluding remarks} \label{sec:conclusion}

Conlon and Lim~\cite{CL21} raised the question of whether the extension of the Valtr grid to higher dimensions
presented in their paper bears any influence on the problem of constructing density-restricted sets with no large convex
subsets in higher dimensions. Here we gave a positive answer and a tight asymptotic bound for $d=3$.
We also obtained the first approximation algorithm for the problem of finding a maximum-size subset of points
in convex position in a density-restricted set in $\RR^3$.
Next, we list a few open questions regarding the remaining gaps and the quality of approximation.

\begin{enumerate}  \itemsep 2pt

\item Is the problem of finding a maximum-cardinality subset in convex position,
  in given finite set in $\RR^3$, still \textrm{NP}-complete for density-restricted sets?

\item Is there a constant-ratio approximation algorithm
  for finding a maximum-size subset in convex position for a given finite set in $\RR^3$?
  Is there one for density-restricted sets?

\end{enumerate}

\later{
If we perturb an $n$-element section of the integer lattice $\ZZ^d$, then we obtain a set of $n$ points in $\RR^d$
with spread $\Omega(n^{1/d})$. However, a small perturbation may substantially increase the \emph{bit-complexity}
of the point set. What is the \emph{minimum} bit-complexity that can accommodate such a perturbation?

\begin{enumerate}  \itemsep 2pt \setcounter{enumi}{2}
\item In constant dimensions $d\in \NN$, what is the smallest $\kappa(d)>0$ such that for every $n\in \NN$
  there exists a set $A$ of $n$ points in an $[n^{\kappa(d)}]^d$ section of the integer lattice $\ZZ^d$
  such that any subset of $A$ in convex position contains at most $\O(n^{(d-1)/(d+1)} \text{polylog}(n)) $ points?
\end{enumerate}
} 

Next are several open questions regarding the size of the largest convex subset
in point sets where the density constraints are relaxed. Let $A$ be a set of $n$ points in general position
in $\RR^d$ satisfying $\Delta(A) \leq \alpha n^\tau$, where $d \geq 2$ and $\alpha,\tau>0$ are constants.
Note that for $\tau \geq 1$, only poly-logarithmic bounds are in effect~\cite{KV03}.

\begin{enumerate}  \itemsep 2pt \setcounter{enumi}{2}

\item Let $d=2$.
  What upper bounds on the size of the largest convex subset can be derived when
$\tau \in \left(\frac12,1\right)$?
 What lower bounds can be derived when $\tau \in \left[\frac34,1\right)$?

 \item  Let $d=3$.
  What upper bounds on the size of the largest convex subset can be derived when
$\tau \in \left(\frac13,1\right)$?
 What lower bounds can be derived when $\tau \in \left[\frac23,1\right)$?

\end{enumerate}

A natural candidate for a lower bound in the third question   
is a suitable perturbation---in the form of the Valtr grid---of a rectangular section of the integer lattice.
Indeed, if $A$ is a $n^\tau \times n^{1-\tau}$ section of this lattice and $\tau>1/2$,
then $\Delta(A) =\O(n^\tau)$.

We conclude  with the following conjecture that generalizes Lemma~\ref{lem:new2}:

\begin{conjecture}\label{con:1}
For every finite set $S\subset \ZZ^d$, $d\geq 2$, and for every integer $t$, $1\leq t\leq V^{1/d}$,
the lattice polytope $\conv(S)$  has $\O\left( \left(V/t^d\right)^{\frac{d-1}{d+1}}\right)$
faces (of any dimension) that contain more than $t$ points in $\ZZ^d$, where $V= \vol(\conv(S))>0$.
\end{conjecture}

Using this technical tool, our upper bound $\gamma_3(n) =\O(n^{1/2})$ in Section~\ref{sec:upper}
would generalize to higher dimensions. Together with the direct generalization
of our lower bound $\gamma_3(n) = \Omega(n^{1/2})$ in Section~\ref{sec:lower},
it would yield an alternative proof of the bound $\gamma_d(n)=\Theta_d \left(n^{\frac{d-1}{d+1}}\right)$ for $d\geq 4$, obtained by Bukh and Dong~\cite{BD22} independently.

\end{document}